\documentclass[12pt]{amsart}

\usepackage{amssymb}

\newtheorem{theorem}{Theorem}[section]
\newtheorem{corollary}[theorem]{Corollary}
\newtheorem{lemma}[theorem]{Lemma}
\newtheorem{proposition}[theorem]{Proposition}

\theoremstyle{definition}
\newtheorem{definition}[theorem]{Definition}

\numberwithin{equation}{section}

\newcommand{\dbar}{\overline{\partial}}
\newcommand{\del}{\partial}
\newcommand{\tensor}{\otimes}

\newcommand{\CC}{\mathbb{C}}
\newcommand{\PP}{\mathbb{P}}
\newcommand{\RR}{\mathbb{R}}

\newcommand{\A}{\mathcal{A}}

\DeclareMathOperator{\End}{End}
\DeclareMathOperator{\GL}{GL}
\DeclareMathOperator{\Hom}{Hom}
\DeclareMathOperator{\id}{Id}
\DeclareMathOperator{\rank}{rank}
\DeclareMathOperator{\SU}{SU}
\DeclareMathOperator{\tr}{tr}
\DeclareMathOperator{\U}{U}
\DeclareMathOperator{\vol}{vol}

\renewcommand{\leq}{\leqslant}

\newcommand*{\longhookrightarrow}{\ensuremath{\lhook\joinrel\relbar\joinrel\rightarrow}}

\begin{document}

\baselineskip=15.2pt

\title{The vortex equation on affine manifolds}

\author[I.\ Biswas]{Indranil Biswas}

\address{School of Mathematics, Tata Institute of Fundamental
Research, Homi Bhabha Road, Bombay 400005, India}

\email{indranil@math.tifr.res.in}

\author[J.\ Loftin]{John Loftin}

\address{Department of Mathematics and Computer Science,
Rutgers University at Newark, Newark, NJ 07102, USA}

\email{loftin@rutgers.edu}

\author[M.~Stemmler]{Matthias Stemmler}

\address{Fachbereich Mathematik und Informatik, Philipps--Universit\"at Marburg, Hans--Meer\-wein--Stra{\ss}e, Lahnberge, 35032 Marburg, Germany}

\email{stemmler@mathematik.uni-marburg.de}

\subjclass[2000]{53C07, 57N16}

\keywords{Affine manifold, vortex equation, stability, dimensional reduction}

\date{}

\begin{abstract}
Let $M$ be a compact connected special affine manifold equipped with an
affine Gauduchon
metric. We show that a pair $(E \,, \phi)$, consisting of a flat vector bundle
$E$ over $M$ and a flat nonzero section $\phi$ of $E$, admits a solution to the
vortex
equation if and only if it is polystable. To prove this, we adapt the
dimensional reduction techniques for holomorphic pairs on K\"ahler
manifolds to the situation of flat pairs on affine manifolds.
\end{abstract}

\maketitle

\section{Introduction}

An affine manifold is a smooth real connected manifold equipped with a
flat torsion--free
connection $D$ on its tangent bundle. Equivalently, an affine structure on an
$n$--dimensional real $C^\infty$
manifold $M$ is provided by an atlas of $M$ such that all the
transition functions are affine maps of the form
\[
x \,\longmapsto\, Ax + b \,, \quad \text{where } A \,\in\, \GL(n \,, \RR) \quad \text{and} \quad b \,\in\, \RR^n \, .
\]
Given an affine manifold $M$, the total space of its tangent bundle $TM$ is
canonically endowed with a complex structure, and the zero section of
$TM
\longrightarrow M$ makes $M$ a totally real submanifold of $TM$. In \cite{Lo09},
a
dictionary was established between the locally constant sheaves on $M$ and the
holomorphic sheaves on $TM$ which are invariant in the fiber directions. In
particular, a flat complex vector bundle over $M$ naturally extends to a
holomorphic vector bundle over $TM$.

An affine manifold is called {\em special\/} if it admits a volume form which is
covariant constant with respect to the flat connection $D$ on $TM$. In
\cite{Lo09}, a Donaldson--Uhlenbeck--Yau type correspondence was established for
flat vector bundles over a compact special affine manifold equipped with an
affine Gauduchon metric. This correspondence states that such a vector bundle
admits an
affine Hermitian--Einstein metric if and only if it is polystable. The proof
of it is an adaptation to the affine situation of the methods of Uhlenbeck and
Yau, \cite{UY86}, \cite{UY89}, for compact K\"ahler manifolds and their
modification by Li and Yau, \cite{LY87}, for the complex Gauduchon case.

A holomorphic pair on a compact K\"ahler manifold $X$ is a pair $(E \,, \phi)$
consisting of a holomorphic vector bundle $E$ over $X$ and a holomorphic section
$\phi$ of $E$ which is not identically equal to zero. These objects were
introduced by Bradlow in \cite{Br90} and \cite{Br91} (see also \cite{GP93},
\cite{GP94a}). Since then they have appeared in various contexts and
turned out to be very useful. For instance,
\begin{itemize}
\item pairs play a very central role in the Donaldson--Thomas theory, and

\item in symplectic topology, pairs yield natural generalizations of
pseudo--holomorphic maps to the equivariant setting \cite{CGMS}.
\end{itemize}

Bradlow defined the notion of $\tau$--stability, where $\tau$ is a
real number, and established a Donaldson--Uhlenbeck--Yau type correspondence for
holomorphic pairs. This correspondence relates $\tau$--stability to the
existence of a Hermitian metric solving the $\tau$--vortex equation, which is
similar to the Hermitian--Einstein equation but additionally involves the
section $\phi$. In \cite{GP94b}, Garc\'ia-Prada showed that the vortex equation
is a dimensional reduction of the Hermitian--Einstein equation for an
$\SU(2)$--equivariant holomorphic vector bundle over $X \times \PP^1_{\mathbb
C}$, where $\SU(2)$ acts trivially on $X$ and in the standard way on
$\PP^1_{\mathbb C}$.

Let $M$ be a compact special affine manifold equipped with an affine Gauduchon
metric. We will call a pair of the form $(E \, ,
\phi)$, where $E$ is a flat vector bundle over $M$ and $\phi$
is a flat nonzero section of $E$, as a flat pair.
Our aim here is to introduce the vortex equation for a flat pair $(E \,,
\phi)$, and to show that $(E \,, \phi)$ admits a solution of the
vortex equation if and only if it is polystable. For this, we first adapt the
theory of Hermitian--Einstein metrics on a flat vector bundle over the affine manifold $M$
to a
smooth complex vector bundle over the product manifold $M \times \PP^1_{\mathbb
C}$ equipped with a certain flat partial connection. Such a vector bundle
canonically
extends to a holomorphic vector bundle over the complex manifold $TM \times
\PP^1_{\mathbb C}$. Then we show that the vortex equation on $M$ is a dimensional
reduction of the Hermitian--Einstein equation on $M \times \PP^1_{\mathbb
C}$.

We obtain the following theorem (see Theorem \ref{main} and Corollary \ref{corollary}):
\begin{theorem}
Let $(M \,, D \,, \nu)$ be an $n$--dimensional compact connected special affine manifold
equipped with an affine Gauduchon metric with associated $(1 \,, 1)$--form
$\omega_M$, and let $(E \,, \phi)$ be a flat pair on $M$. Let $\tau$ be a real number, and let
\[
\widehat \tau \,=\, \frac{\tau}{2} \int_M \frac{\omega_M^n}{\nu} \, .
\]
Then $E$ admits a smooth Hermitian metric satisfying the $\tau$--vortex equation if and only if it is $\widehat \tau$--polystable.
\end{theorem}

\medskip
\noindent
\textbf{Acknowledgement:}\, We are very grateful to the referee for comments
to improve the exposition. The first author wishes to thank ICMAT, Madrid,
for hospitality while a part of the work was carried out. The second author
gratefully acknowledges support from Simons Collaboration Grant for Mathematicians 210124.

\section{Preliminaries}

\subsection{Affine manifolds}

Let $(M \,, D)$ be an affine manifold of dimension $n$, meaning that $D$ is a
flat torsion--free connection on the tangent bundle $TM$ of a real $C^\infty$
manifold $M$ of dimension $n$. Throughout the paper, all manifolds are assumed to
be connected and $C^\infty$.
Given an atlas on $M$ such that all the transition maps are affine transformations, the
corresponding coordinates $\{ x^i \}$ are called {\em local affine
coordinates}. If $\{ x^i \}$ is defined on the open subset $U \subset M$,
then write $y^i$ for the fiber coordinates corresponding to the local
trivialization of the tangent bundle given by $\left\{ \frac{\del}{\del x^i}
\right\}_{i=1}^n$. Then on the open subset $TU \subset TM$, we have the
holomorphic coordinate functions $z^i := x^i + \sqrt{-1} \, y^i$, turning $TM$
into a complex manifold in a natural way. This complex manifold of dimension $n$
will be denoted by $M^\CC$. The zero section of $TM \longrightarrow M$ makes $M$
a totally real submanifold of $M^\CC$.

The vector bundle of $(k \,, l)$--forms on $M$ is defined as
\[
\A^{k,l}(M) \,:=\, \bigwedge\nolimits^k T^\ast M \,\tensor\, \bigwedge\nolimits^l T^\ast M \, ;
\]
these forms are restrictions of $(k \,, l)$--forms on the complex
manifold $M^\CC$. There are differential operators
\begin{align*}
\del  \,:=\, \frac{1}{2} \, (d \tensor \id)\,: \, \A^{k,l}(M)
\,&\longrightarrow\, \A^{k+1,l}(M) \, , \\
\dbar \,:=\, (-1)^k \frac{1}{2} \, (\id \tensor d)\,: \, \A^{k,l}(M)
\,&\longrightarrow\, \A^{k,l+1}(M) \, ,
\end{align*}
which are the restrictions of the corresponding operators on $M^\CC$. Also,
there is a wedge product on the direct sum of $(k \, , l)$--forms on $M$,
which is the restriction of the wedge product on $M^\CC$; see \cite{Lo09}.

The affine manifold $M$ is called {\em special\/} if it admits a volume form (meaning a non--vanishing top--degree form) $\nu$ which is covariant constant with respect to the flat connection $D$ on $TM$.

On a special affine manifold $(M \,, D \,, \nu)$, the volume form $\nu$ induces
homomorphisms
\begin{alignat*}{2}
\A^{n,l}(M) &\,\longrightarrow\, \bigwedge\nolimits^l T^\ast M, \quad & \nu \tensor \chi &\,\longmapsto\, (-1)^{\frac{n(n-1)}{2}} \, \chi \, , \\
\A^{k,n}(M) &\,\longrightarrow\, \bigwedge\nolimits^k T^\ast M, \quad & \chi
\tensor \nu &\,\longmapsto\, (-1)^{\frac{n(n-1)}{2}} \, \chi \, ;
\end{alignat*}
these homomorphisms will be called \textit{division\/} by $\nu$. If
$M$ is compact,
an $(n \,, n)$--form $\chi$ on $M$ can be integrated by considering the integral
\[
\int_M \frac{\chi}{\nu} \, .
\]

A smooth Riemannian metric $g$ on $M$ gives rise to a $(1 \,, 1)$--form expressed in local affine coordinates as
\[
\omega \,=\, \sum_{i,j=1}^n g_{ij} \, dx^i \tensor dx^j \, ;
\]
it is the restriction of the corresponding $(1 \,, 1)$--form on $M^\CC$ given by
the extension of $g$ to $M^\CC$. The metric $g$ is called an {\em affine
Gauduchon metric\/} if
\[
\partial \dbar (\omega^{n-1}) \,=\, 0
\]
(recall that $n$ is the dimension of $M$). By \cite[Theorem 5]{Lo09}, on a
compact connected special affine manifold, every conformal class of Riemannian
metrics
contains an affine Gauduchon metric, which is unique up to a positive scalar.

In the context of affine manifolds, the right analogue of a
holomorphic vector bundle over a complex manifold is a flat complex vector
bundle. To explain this, let $E$ be a smooth complex
vector bundle over an affine manifold $M$. The pullback of $E$ to
$M^\CC$ by the natural projection $M^\CC = TM \longrightarrow M$
will be denoted by $E^\CC$. The transition functions of $E^\CC$ are
obtained by extending the transition functions of $E$ in a constant
way along the fibers of $TM$. Such a transition function on $M^\CC$
is holomorphic if and only if the corresponding transition function for
$E$ is locally constant. Consequently,
$E^\CC$ is a holomorphic vector bundle over $M^\CC$ if and only if
$E$ is a flat vector bundle over $M$. Therefore, the map $E \longmapsto E^\CC$ gives
a bijective correspondence between flat vector bundles on $M$ and
holomorphic vector bundles on $M^\CC$ that are constant along the fibers of $TM$.
Since $E^\CC$ is the pullback of a vector bundle on $M$, ``constant along
the fibers of $TM$'' is well-defined.

Let $(E \,, \nabla)$ be a flat complex vector bundle over $M$, meaning $E$ is a
smooth complex vector bundle and $\nabla$ is a flat connection on $E$.
A Hermitian metric $h$ on $E$ defines a Hermitian metric on $E^\CC$. Let $d^h$
be the Chern connection associated to this Hermitian metric on the
holomorphic vector bundle $E^\CC$. Then $d^h$ corresponds to a pair
\[
(\del^h \,, \dbar) \,=\, (\del^{h,\nabla} \,, \dbar^{\nabla}) \, ,
\]
where
\[
\del^{h,\nabla}: \, E \,\longrightarrow\, \A^{1,0}(E) \quad \text{and} \quad
\dbar^{\nabla}:  \, E \,\longrightarrow\, \A^{0,1}(E)
\]
are smooth differential operators. Here we write $\A^{k,l}(E) := \A^{k,l}(M) \tensor E$. This pair $(\del^h \,, \dbar)$ is called the {\em extended Hermitian connection\/} of $(E \,, h)$. Similarly, there are locally defined {\em extended connection forms}
\[
\theta \,\in\, C^\infty(\A^{1,0}(\End E)) \, ,
\]
an {\em extended curvature form}
\[
  R \,=\, \dbar \theta \,\in\, C^\infty(M \,, \A^{1,1}(\End E)) \, ,
\]
an {\em extended mean curvature}
\[
  K \,=\, \tr_g R \,\in\, C^\infty(M \,, \End E) \, ,
\]
and an {\em extended first Chern form}
\[
  c_1(E \,, h) \,=\, \tr R \,\in\, C^\infty(M \,, \A^{1,1}(M)) \, ,
\]
which are the restrictions of the corresponding objects on $E^\CC$. Here $\tr_g$
denotes contraction of differential forms using the Riemannian metric $g$, and
$\tr$ denotes the trace homomorphism on the fibers of $\End E$.

The extended first Chern form is given by
\[
  c_1(E \,, h) \,=\, - \del \dbar (\log \det (h_{\alpha \bar \beta})) \, ,
\]
where $h_{\alpha \bar \beta} = h(s_\alpha, s_\beta)$ in a locally constant frame $\{ s_\alpha \}$ of $E$.

The extended first Chern form and the extended mean curvature are related by
\[
  (\tr K) \, \omega^n \,=\, n \, c_1(E \,, h) \wedge \omega^{n-1} \, .
\]

\begin{definition}
A Hermitian metric $h$ on $E$ is called a {\em Hermitian--Einstein metric\/}
(with respect to $g$) if its extended mean curvature $K_h$ is of the form
\[
  K_h \,=\, \gamma \cdot \id_E
\]
for some real constant $\gamma$.
\end{definition}

The {\em degree\/} of $(E \,, \nabla)$ with respect to a Gauduchon metric $g$ on $M$ is defined to be
\[
  \deg_g(E) \,:=\, \int_M \frac{c_1(E \,, h) \wedge \omega^{n-1}}{\nu} \, ;
\]
it is well--defined by \cite[p.\ 109]{Lo09}.

As usual, if $\rank(E) > 0$, the {\em slope\/} of $E$ with respect to $g$ is
defined to be
\[
  \mu_g(E) \,:=\, \frac{\deg_g(E)}{\rank(E)} \, .
\]

\begin{definition} \mbox{}
\begin{enumerate}
\item[(i)] $(E \,, \nabla)$ is called {\em stable\/} (with respect to $g$) if for every proper nonzero flat subbundle $E'$ of $E$ we have
\[
  \mu_g(E') \,<\, \mu_g(E) \, .
\]
\item[(ii)] $(E \,, \nabla)$ is called {\em polystable\/} (with respect to $g$) if
\[
  (E \,, \nabla) \,=\, \bigoplus_{i=1}^N \, (E^i \,, \nabla^i)\, ,
\]
where each pair $(E^i \,, \nabla^i)$ is a stable flat vector bundle with
slope $\mu_g(E^i) = \mu_g(E)$.
\end{enumerate}
\end{definition}

In \cite{Lo09}, the following Donaldson--Uhlenbeck--Yau type correspondence was
established.

\begin{theorem}[{\cite[Theorem 1]{Lo09}}] \label{loftin-thm}
Let $(M \,, D \,, \nu)$ be a compact special affine manifold equipped with an
affine Gauduchon metric $g$, and let $(E \,, \nabla)$ be flat complex vector
bundle over $M$. Then $E$ admits a Hermitian--Einstein metric with respect to
$g$ if and only if it is polystable.
\end{theorem}

Since we rely on these techniques below, we summarize below the main ideas
of the proof.

\begin{proof}[Outline of proof]
The proof is an adaptation to the affine situation of the techniques of
Uhlenbeck--Yau for holomorphic vector bundles over compact K\"ahler manifolds,
\cite{UY86}, and their extension by Li--Yau to vector bundles over compact complex
Gauduchon manifolds \cite{LY87}.  In particular, we have set things up so that
all the relevant quantities on $(M\, ,E)$ such as the metric $g$, the extended
Hermitian connection $(\del^h\, ,\dbar)$, etc., are restrictions of the same
quantities on the holomorphic vector bundle $(M^\CC\, ,E^\CC)$, with the
quantities on $M^\CC$ being constant along the fibers of $M^\CC\,\longrightarrow
\, M$. The idea of the proof is to think of all the calculations as happening upstairs on the noncompact $M^\CC$ while still managing to integrate over the compact manifold $M$.

Here are a few more details. The proof in the complex case relies on most of 
the standard tools of the elliptic theory on compact manifolds: integration 
by parts, the maximum principle, $L^p$ estimates, Sobolev embedding, 
spectral theory of elliptic operators, and some intricate local 
calculations.  Our setup forces the local calculations to be \emph{exactly} 
the same as in the complex case.  The maximum principle, $L^p$ estimates, 
Sobolev embedding, and spectral theory translate to our case with no 
difficulty.  The main innovation is to handle integration by parts. For this 
we need the definition of integrating an $(n\, ,n)$ form $\chi$ on $M$ via 
$\int_M \frac{\chi}{\nu}$ as above. The fact that $D\nu\,=\,0$ ensures that 
integrating by parts does not produce any extraneous terms, and so the local 
calculations on $M$ remain in exact correspondence with those on $M^\CC$.
\end{proof}

\subsection{Partial connections}

In Section \ref{sec3}, we will adapt some of the notions from the theory of
affine manifolds to products of affine manifolds with the complex projective
line $\PP^1_{\mathbb C}$; these products are non--affine smooth real manifolds.
For this, we need partial connections on the product $M\times \PP^1_{\mathbb C}$, for the affine directions on $M$ and the complex directions on $\PP^1$ must be distinguished. We recall the definition of partial connections.

Let $X$ be a smooth real manifold; its real tangent
bundle will be denoted by $T_\RR X$. Let
\[
  S \,\subset\, T_\CC X \,:=\, T_\RR X \tensor \CC
\]
be a subbundle of positive rank which is {\em integrable}, meaning
\begin{itemize}
\item $S \cap {\overline S}\, \subset\, T_\CC X$ has constant rank, and

\item both $S$ and $S + \overline{S}$ are closed under the Lie bracket
(the first condition implies that $S + \overline{S}$ is a subbundle of
$T_\CC X$).
\end{itemize}
Let
\begin{equation} \label{projection}
  q_S\,: \, T_\CC^\ast X \,:=\, (T_\CC X)^\ast \,\longrightarrow\, S^\ast
\end{equation}
be the dual of the inclusion map of $S$ in $T_\CC X$.

Let $E$ be a smooth complex vector bundle over $X$. A {\em partial connection
on $E$ in the direction of $S$\/} is a smooth differential operator
\[
  \nabla: \, E \,\longrightarrow\, S^\ast \tensor E
\]
satisfying the Leibniz condition, meaning that for a smooth function $f$ on $X$
and a smooth section $s$ of $E$, the identity
\[
  \nabla(f s) \,=\, f \nabla(s) + q_S(df) \tensor s
\]
holds, where $q_S$ is the projection in \eqref{projection}.

Since the distribution $S$ is integrable, smooth sections of the kernel of
$q_S$ (see \eqref{projection}) are closed under the exterior derivation. Therefore, there is an induced exterior derivation
\begin{equation} \label{d-hat}
  \widehat d: \, C^\infty (X,\, S^\ast) \,\longrightarrow\,
C^\infty \left(X,\, \bigwedge\nolimits^2 S^\ast\right)
\end{equation}
on the smooth sections of $S^\ast$.

Let $\nabla$ be a partial connection on $E$ in the direction of $S$. Consider the differential operator
\[
  \nabla_1: \, S^\ast \tensor E \,\longrightarrow\, \left(\bigwedge\nolimits^2 S^\ast\right) \tensor E
\]
defined by
\[
  \nabla_1(\chi \tensor s) \,=\, \widehat d(\chi) \tensor s - \chi \wedge \nabla(s) \, ,
\]
where $\widehat d$ is constructed in \eqref{d-hat}. The composition
\[
  E \,\stackrel{\nabla}{\longrightarrow}\, S^\ast \tensor E \,\stackrel{\nabla_1}{\longrightarrow}\, \left(\bigwedge\nolimits^2 S^\ast\right) \tensor E
\]
is $C^\infty(X)$--linear and thus defines a smooth section
\[
  R(\nabla) \,\in\, C^\infty\left(X \,, \left(\bigwedge\nolimits^2 S^\ast\right) \tensor E \tensor E^\ast\right)
  \,=\, C^\infty\left(X \,, \left(\bigwedge\nolimits^2 S^\ast\right) \tensor \End(E)\right) \, .
\]
This section $R(\nabla)$ is called the {\em curvature\/} of $\nabla$. If
$R(\nabla) = 0$, then the partial connection $\nabla$ is called {\em flat}. By
\cite[Theorem 1]{Ra79}, a partial connection $\nabla$ on $E$ is flat if and only
if $E$ admits locally defined smooth frames $\{ s_\alpha \}$ satisfying
$\nabla(s_\alpha) = 0$.

A pair $(E \,, \nabla)$ consisting of a smooth complex vector bundle $E$ over
$X$ and a flat partial connection $\nabla$ on $E$ in the direction of $S$ will
be called an {\em $S$--partially flat vector bundle}. We also write $E$ for $(E
\,,
\nabla)$ if $\nabla$ is clear from the context.

\section{Hermitian--Einstein metrics over $M \times \PP^1_{\mathbb
C}$}\label{sec3}

In this section, we investigate Hermitian--Einstein metrics on bundles over $M\times \PP^1_{\mathbb C}$.  We will use this set--up below to address the vortex equation by adapting the dimensional reduction technique of Garc\'ia-Prada \cite{GP94b} to this case.

Let $(M \,, D)$ be an affine manifold of dimension $n$. Denote by
$\PP^1 \,=\, \PP^1_{\mathbb C}$ the complex projective line. Consider the
product manifold
\[
  X \,:=\, M \times \PP^1 \, ,
\]
which is a smooth real manifold of dimension $n+2$. Let
\begin{equation}\label{pq}
  p: \, M \times \PP^1 \,\longrightarrow\, M \quad \text{and} \quad
  q: \, M \times \PP^1 \,\longrightarrow\, \PP^1
\end{equation}
be the natural projections.
Recall the idea from the proof of Theorem \ref{loftin-thm} above.  We will find a dictionary between geometric objects on the compact manifold $X$ and geometric objects on $M^\CC \times \PP^1$ which are constant along the fibers of the projection from $M^\CC \times \PP^1 \to X=M\times \PP^1$.  Our goal is to define structures on $X$ so that the local calculations and integration by parts needed to prove the Donaldson--Uhlenbeck--Yau correspondence are formally the same as on the complex manifold $M^\CC \times \PP^1$, but all the integration can be carried out on $X$.

The complexified tangent bundle of $X$ can be decomposed as
\[
  T_\CC X \;=\; p^\ast T_\CC M \oplus q^\ast T_\CC \PP^1
  \;=\; p^\ast T_\CC M \oplus q^\ast T^{1,0} \PP^1 \oplus q^\ast T^{0,1} \PP^1
\, .
\]
Here $T^{1,0} \PP^1$ and $T^{0,1} \PP^1$ are respectively the holomorphic and
anti--holomorphic tangent bundles of $\PP^1$.

The two distributions
\begin{equation}\label{dds}
  S^{1,0} \,:=\, p^\ast T_\CC M \oplus q^\ast T^{1,0} \PP^1 \quad \text{and} \quad
  S^{0,1} \,:=\, p^\ast T_\CC M \oplus q^\ast T^{0,1} \PP^1
\end{equation}
are integrable. A smooth complex vector bundle $E$ over $X$ admits a flat partial connection in the direction of $S^{0,1}$ if and only if it admits local trivializations with transition functions $\varphi$ satisfying
\[
  q_{S^{0,1}}(d \varphi) \,=\, 0 \, ,
\]
where $q_{S^{0,1}}$ is defined as in \eqref{projection} for $S^{0,1}$ in
\eqref{dds}. This means that $\varphi$
is locally constant in the direction of $M$ and holomorphic in the direction of $\PP^1$. Denote by $E^\CC$ the pullback of $E$ to $M^\CC \times \PP^1$ by the natural projection
\[
  M^\CC \times \PP^1 \,=\, TM \times \PP^1 \,\longrightarrow\, M \times \PP^1 \,=\, X \, .
\]
The transition functions for $E^\CC$ are obtained by extending the transition
functions of $E$ in a constant way along the fibers of $TM$. Consequently,
$E^\CC$ is a holomorphic vector bundle if and only if $E$ is an
$S^{0,1}$--partially flat vector bundle. Therefore, the map $E \longmapsto
E^\CC$ gives
a bijective correspondence between $S^{0,1}$--partially flat vector bundles on
$X$ and
holomorphic vector bundles on $M^\CC \times \PP^1$ that are constant along the
fibers
of $TM$.

We define $(k \,, l)$--forms on $X$ to be smooth sections of the vector bundle
\[
  \A^{k,l}(X) \,:=\, \bigwedge\nolimits^k (S^{1,0})^\ast
  \,\tensor\, \bigwedge\nolimits^{l} (S^{0,1})^\ast \, ;
\]
these forms are restrictions of $(k \,, l)$--forms on the complex manifold
$M^\CC \times \PP^1$. Just as in the affine case, there are natural $\del$ and
$\dbar$ operators on these forms which are the restrictions of the corresponding
operators on $M^\CC \times\PP^1$. More precisely, denote by $d^{1,0}$
(respectively, $d^{0,1}$) the differential operator given in \eqref{d-hat} for
the distribution $S^{1,0}$ (respectively, $S^{0,1}$) in \eqref{dds}. The induced
operators
\begin{align*}
  d^{1,0}: \, C^\infty\left(M\times \PP^1,\, \bigwedge\nolimits^k (S^{1,0})^\ast\right)
\,&\longrightarrow\, C^\infty\left(M\times \PP^1,\, \bigwedge\nolimits^{k+1}
(S^{1,0})^\ast\right) \, ,
\\
  d^{0,1}: \, C^\infty\left(M\times \PP^1,\, \bigwedge\nolimits^l (S^{0,1})^\ast\right)
\,&\longrightarrow\, C^\infty\left(M\times \PP^1,\, \bigwedge\nolimits^{l+1}
(S^{0,1})^\ast\right)
\end{align*}
will be denoted by the same symbols. Then the operators $\del$ and $\dbar$ are defined by
\begin{align*}
  \del  \,=\, \frac{1}{2} \, (d^{1,0} \tensor \id): \, \A^{k,l}(X) \,&\longrightarrow\, \A^{k+1,l}(X) \, , \\
  \dbar \,=\, (-1)^k \frac{1}{2} \, (\id \tensor d^{0,1}): \, \A^{k,l}(X) \,&\longrightarrow\, \A^{k,l+1}(X) \, .
\end{align*}

The wedge product for $(k \,, l)$--forms on $X$ is defined in the same way as in \cite{Lo09}; more precisely,
\[
  (\chi_1 \tensor \psi_1) \wedge (\chi_2 \tensor \psi_2) \,:=\, (-1)^{l_1 k_2} \, (\chi_1 \wedge \chi_2) \tensor (\psi_1 \wedge \psi_2)
\]
if $\chi_i \tensor \psi_i$ are forms of type $(k_i \,, l_i)$, $i = 1, 2$.

Now let $(M \,, D \,, \nu)$ be a compact special affine manifold, meaning that
$(M \,, D)$ is an affine manifold equipped with a $D$--covariant constant
volume form $\nu$. Let $g_{\PP^1}$ be the Fubini--Study metric on $\PP^1$ with
K\"ahler form $\omega_{\PP^1}$, normalized so that $$\int_{\PP^1} \omega_{\PP^1}
\,=\, 1\, .$$

A {\em Hermitian metric\/} on $X$ is defined to be a Riemannian metric $g$ on $X$ of the form
\[
  g \,=\, p^\ast g_M \oplus q^\ast g_{\PP^1} \, ,
\]
where $g_M$ is a Riemannian metric on $M$. Such a metric $g$ gives rise to a
$(1 \,, 1)$--form $\Omega_g$ on $X$, which is the restriction of the $(1
\,, 1)$--form on $M^\CC \times \PP^1$ given by the extension of $g$ to $M^\CC
\times \PP^1$; it has the following expression:
\[
  \Omega_g \,=\, p^\ast \omega_M - \sqrt{-1} \, q^\ast \omega_{\PP^1} \, ,
\]
where $\omega_M$ is the $(1 \,, 1)$--form on $M$ corresponding to $g_M$. If
\[
  \del \dbar(\Omega_g^n) \,=\, 0 \, ,
\]
then $g$ is called a {\em Gauduchon metric}. Since $\dim_\RR M = n$ and $\dim_\RR \PP^1 = 2$, we have
\[
  \del \dbar(\Omega_g^n) \,=\, p^\ast(\del \dbar(\omega_M^n)) - \sqrt{-1} \, n p^\ast (\del \dbar(\omega_M^{n-1})) \wedge q^\ast \omega_{\PP^1}
 \,=\, - \sqrt{-1} \, n p^\ast (\del \dbar(\omega_M^{n-1})) \wedge q^\ast
\omega_{\PP^1} \, .
\]
Therefore, $g$ is a Gauduchon metric on $X$ if and only if $g_M$ is an affine
Gauduchon metric on~$M$. (Note this construction depends only on the fact that $g_{\PP^1}$ is K\"ahler.)

Since
\[
\bigwedge\nolimits^{n+1} (S^{0,1})^\ast \,\cong\, p^\ast \left(\bigwedge
\nolimits^n T_\CC^\ast M\right) \tensor q^\ast (T^{0,1} \PP^1)^\ast\, ,
\]
and $p^\ast \nu$ is a non--vanishing section of
$p^\ast\left(\bigwedge\nolimits^n T_\CC^\ast M\right)$, for every $0\,\leq\, k\,
\leq \, n+1$, we have a map
\begin{align*}
  \A^{k,n+1}(X) \,=\, \bigwedge\nolimits^k (S^{1,0})^\ast \tensor \bigwedge\nolimits^{n+1} (S^{0,1})^\ast
  \,&\longrightarrow\, \bigwedge\nolimits^{k+1} T_\CC^\ast X \, , \\
 \chi \tensor (p^\ast \nu \tensor \psi) \,&\longmapsto\, (-1)^{\frac{n(n+1)}{2}}
\, \chi \wedge (\sqrt{-1} \, \psi) \, .
\end{align*}
Here $\chi$ is a smooth section of $\bigwedge\nolimits^k
(S^{1,0})^\ast$, and $\psi$ is a smooth section
of $q^\ast (T^{0,1} \PP^1)^\ast$. On the right--hand side, $\chi$ (respectively, $\psi$) is considered as a $k$--form (respectively, $1$--form) on $X$ via the inclusion
\[
  \bigwedge\nolimits^k (S^{1,0})^\ast \,\longhookrightarrow\, \bigwedge\nolimits^k T_\CC^\ast X \quad
  \left(\text{respectively, } q^\ast (T^{0,1} \PP^1)^\ast \,\longhookrightarrow\, T_\CC^\ast X \right) \, .
\]
Similarly, we have a map
\begin{align*}
  \A^{n+1,l}(X) \,=\, \bigwedge\nolimits^{n+1} (S^{1,0})^\ast \tensor \bigwedge\nolimits^l (S^{0,1})^\ast
  \,&\longrightarrow\, \bigwedge\nolimits^{l+1} T_\CC^\ast X \, , \\
  (p^\ast \nu \tensor \psi) \tensor \chi \,&\longmapsto\, - (-1)^{\frac{n(n+1)}{2}} \, \chi \wedge (\sqrt{-1} \, \psi) \, .
\end{align*}
Both of these maps are called \textit{division\/} by $p^\ast \nu$. The factor
$\sqrt{-1}$
ensures that for every Hermitian metric $g$ on $X$, the form
\[
  \frac{\Omega_g^{n+1}}{p^\ast \nu}
\]
is real and thus it is a volume form; the factor $(-1)^{n(n+1)/2}$ ensures that
the form induces the same orientation on $X$ as the volume form
\[
  p^\ast \nu \wedge (- \sqrt{-1} \, q^\ast \omega_{\PP^1}) \, .
\]
Also, note that for $k=l=n+1$, the two maps coincide. An $(n+1 \,, n+1)$--form
$\chi$ on $X$ can be integrated by considering the integral
\[
  \int_X \frac{\chi}{p^\ast \nu} \, .
\]

As in \cite[Proposition 3]{Lo09}, we have the following proposition, which plays
an important role when integrating by parts on $X$. The proof is identical to that
of Proposition 3 in \cite{Lo09}.

\begin{proposition} \label{int-parts}
For an $(n \,, n+1)$--form $\chi$ on $X$, the identity
\[
 \frac{\del \chi}{p^\ast \nu} \,=\, \frac{1}{2} \, d\left(\frac{\chi}{p^\ast
\nu}\right)
\]
holds, while for an $(n+1 \,, n)$--form $\chi$ on $X$,
\[
  \frac{\dbar \chi}{p^\ast \nu} \,=\, (-1)^{n+1} \frac{1}{2} \, d\left(\frac{\chi}{p^\ast \nu}\right) \, .
\]
\end{proposition}

Let $(E \,, \nabla)$ be an $S^{0,1}$--partially flat vector bundle on $X$, and
let $h$ be a smooth Hermitian metric on $E$. As mentioned above, $E$ extends to
a holomorphic vector bundle $E^\CC$ over $M^\CC \times \PP^1$. The metric $h$
defines a Hermitian metric on $E^\CC$; let $d^h$ denote the corresponding Chern
connection on $E^\CC$. As in the affine case, $d^h$ corresponds to a pair
\[
  (\del^h \,, \dbar) \,=\, (\del^{h,\nabla} \,, \dbar^{\nabla}) \, ,
\]
where
\[
  \del^{h,\nabla}: \, E \,\longrightarrow\, \A^{1,0}(E) \quad \text{and} \quad
  \dbar^{\nabla}:  \, E \,\longrightarrow\, \A^{0,1}(E)
\]
are smooth differential operators. We write $\A^{k,l}(E) \,:=\,
\A^{k,l}(X) \tensor E$ as before. This pair $(\del^h \,, \dbar)$ is called the
{\em extended Hermitian connection\/} for $(E \,, h)$. Similarly, there are
locally defined {\em extended connection forms}
\[
  \theta \,\in\, C^\infty(\A^{1,0}(\End E)) \, ,
\]
an {\em extended curvature form}
\[
  R \,=\, \dbar \theta \,\in\, C^\infty(X \,, \A^{1,1}(\End E)) \, ,
\]
an {\em extended mean curvature}
\[
  K \,=\, \tr_g R \,\in\, C^\infty(X \,, \End E) \, ,
\]
and an {\em extended first Chern form}
\[
  c_1(E \,, h) \,=\, \tr R \,\in\, C^\infty(X \,, \A^{1,1}(X)) \, ,
\]
which are the restrictions of the corresponding objects on $E^\CC$. Here
$\tr_g$ denotes contraction of differential forms using the Riemannian metric
$g$, and $\tr$ as before denotes the trace homomorphism on the fibers of $\End
E$.

The extended first Chern form is given by
\[
  c_1(E \,, h) \,=\, - \del \dbar (\log \det (h_{\alpha \bar \beta})) \, ,
\]
where $h_{\alpha \bar \beta}\,:=\,h(s_\alpha, s_\beta)$ with respect to a
locally defined smooth frame $\{ s_\alpha \}$ of $E$ satisfying
$\nabla(s_\alpha)\,=\, 0$.

The extended first Chern form and the extended mean curvature are related by
the equation
\[
  (\tr K) \, \Omega_g^{n+1} \,=\, (n+1) \, c_1(E \,, h) \wedge \Omega_g^n \, .
\]

\begin{definition}
A Hermitian metric $h$ on $E$ is called a {\em Hermitian--Einstein metric\/} (with respect to $g$) if its extended mean curvature $K_h$ is of the form
\[
  K_h \,=\, \gamma \cdot \id_E
\]
for some real constant $\gamma$.
\end{definition}

The {\em degree\/} of $(E \,, \nabla)$ with respect to a Gauduchon metric $g$ on $X$ is defined to be
\[
  \deg_g(E) \,:=\, \int_X \frac{c_1(E \,, h) \wedge \Omega_g^n}{p^\ast \nu} \, ;
\]
as in the affine case, it is independent of the choice of Hermitian metric $h$
by Proposition~\ref{int-parts} because $g$ is Gauduchon.

If $\rank(E) > 0$, the {\em slope\/} of $E$ with respect to $g$ is defined to be
\[
  \mu_g(E) \,:=\, \frac{\deg_g(E)}{\rank(E)} \, .
\]

\begin{definition} \mbox{}
\begin{enumerate}
\item[(i)] $(E \,, \nabla)$ is called {\em stable\/} (with respect to $g$) if
for every proper nonzero subbundle $E'$ of $E$ which is preserved by
$\nabla$, meaning $\nabla(E') \subset \A^{0,1}(E')$, the inequality
\[
  \mu_g(E') \,<\, \mu_g(E)
\]
holds.

\item[(ii)] $(E \,, \nabla)$ is called {\em polystable\/} (with respect to $g$) if
\[
  (E \,, \nabla) \,=\, \bigoplus_{i=1}^N \, (E^i \,, \nabla^i)\, ,
\]
where $(E^i \,, \nabla^i)$ are $S^{0,1}$--partially flat stable bundles
with slope $\mu_g(E^i) \,=\, \mu_g(E)$.
\end{enumerate}
\end{definition}

Given this set--up, the following Donaldson--Uhlenbeck--Yau type correspondence
can be proved in the same way as in the affine case (see \cite[Theorem 1]{Lo09}).

\begin{theorem}\label{duy}
Let $(M \,, D \,, \nu)$ be a compact connected special affine manifold equipped
with an affine Gauduchon metric $g_M$, and let $X \,:=\, M \times \PP^1$
together with the Gauduchon metric $g$ be as defined above. Let $(E \,, \nabla)$
be an  $S^{0,1}$--partially flat vector bundle over $X$. Then $E$ admits a
Hermitian--Einstein metric with respect to $g$ if and only if it is polystable.
\end{theorem}

\begin{proof}[Outline of proof]
As in the proof of Theorem \ref{loftin-thm} above, we have been careful
to ensure that the objects we define on $X$ correspond exactly to objects
on the complex manifold $M^\CC\times\PP^1$ which are constant along the
fibers of the projection $M^\CC \times \PP^1\,\longrightarrow\, X$.  Thus we
may mimic the proof of the Donaldson--Uhlenbeck--Yau correspondence in the complex case as long as we check two things:
\begin{itemize}
\item The local calculations on $X$ correspond exactly to calculations on the complex manifold $M^\CC \times \PP^1$.
\item Integration by parts works.
\end{itemize}
Above, we deal with the first of these by introducing partial connections on $X$. The integration by parts also follows above since $\omega_{\PP^1}$ is K\"ahler and $\omega_M$ is Gauduchon.
\end{proof}

\section{Dimensional reduction}

Let $(M \,, D \,, \nu)$ be a compact special affine manifold of dimension $n$
equipped with an affine Gauduchon metric $g_M$.

\begin{definition}
A {\em flat pair\/} on $M$ is a pair $((E \,, \nabla_E) \,, \phi)$ (or $(E \,,
\phi)$ if $\nabla_E$ is understood from the context) consisting of a flat
complex vector bundle $(E \,, \nabla_E)$ over $M$, and a nonzero flat section
$\phi$ of $E$.
\end{definition}

\begin{definition}
Let $(E \,, \phi)$ be a flat pair on $M$, and let $\tau$ be a real number.
\begin{enumerate}
\item[(i)] $(E \,, \phi)$ is called {\em $\tau$--stable\/} (with respect to
$g_M$) if the following two conditions are satisfied:
\begin{itemize}
\item $\mu_g(E') < \tau$ for every flat subbundle $E'$ of $E$ with $\rank(E') > 0$.
\item $\mu_g(E/E') > \tau$ for every flat subbundle $E'$ of $E$ with $0 < \rank(E') < \rank(E)$ containing the image of the section $\phi$.
\end{itemize}
\item[(ii)] $(E \,, \phi)$ is called {\em $\tau$--polystable\/} (with respect
to $g_M$) if it is either $\tau$--stable or $E$ decomposes as a direct sum of
flat subbundles
\[
  E \,=\, E' \oplus E''
\]
such that $\phi$ is a section of $E'$, the flat pair $(E' \,,
\phi)$ is $\tau$--stable, and the flat vector bundle $E''$ is polystable with
slope $\mu_g(E'') = \frac{\tau}{n}$, where $n\, =\, \dim M$.
\end{enumerate}
\end{definition}

\begin{definition}
Given a flat pair $(E \,, \phi)$ on $M$ and a real number $\tau$, a smooth Hermitian metric $h$ on $E$ is said to satisfy the {\em $\tau$--vortex equation\/} if
\begin{equation} \label{vortex}
  K_h + \frac{1}{2} \, \phi \circ \phi^\ast - \frac{\tau}{2} \, \id_E \,=\, 0 \, ,
\end{equation}
where $K_h$ is the extended mean curvature of $(E \,, h)$, $\phi$ is regarded as a homomorphism from the trivial Hermitian line bundle on $M$ to $E$, and $\phi^\ast$ denotes its adjoint with respect to $h$.
\end{definition}

As mentioned above, a flat complex vector bundle $E$ over $M$ corresponds to a holomorphic vector bundle $E^\CC$ over $M^\CC$ which is constant along the fibers of $M^\CC \,=\, TM \longrightarrow M$. A nonzero flat section $\phi$ of $E$ (respectively, a smooth Hermitian metric $h$ on $E$) corresponds to a nonzero holomorphic section $\phi^\CC$ of $E^\CC$ (respectively, a smooth Hermitian metric $h^\CC$ on $E^\CC$) which is constant along the fibers of $TM$. Since the extended mean curvature $K_h \,=\, \tr_g R_h$ is the restriction to $M$ of the usual mean curvature of $h^\CC$ on $E^\CC$, the Hermitian metric $h$ satisfies the $\tau$--vortex equation \eqref{vortex} if and only if $h^\CC$ satisfies the usual $\tau$--vortex equation (see \cite[(2.6c)]{Br90}) for the holomorphic pair $(E^\CC \,, \phi^\CC)$ over $M^\CC$.

We can now state our main result.

\begin{theorem} \label{main}
Let $(M \,, D \,, \nu)$ be a compact connected special affine manifold equipped
with an affine Gauduchon metric $g_M$, and let $(E \,, \phi)$ be a flat pair on
$M$. Let $\tau$ be a real number; define
\[
  \widehat \tau \,:=\, \frac{\tau \cdot \vol(M)}{2} \, ,
\]
where $\vol(M) \,=\, \int_M \frac{\omega_M^n}{\nu}$ is the volume of $M$
with respect to $g_M$. If $(E \,, \phi)$ is $\widehat \tau$--stable, then there
exists a smooth Hermitian metric on $E$ satisfying the $\tau$--vortex equation.
\end{theorem}

The proof of Theorem \ref{main} relies on the technique of {\em dimensional
reduction}, which was developed in the K\"ahler case by Garc\'{i}a-Prada (see
\cite{GP94b}). We will now extend it to our context.

Define $X \,:=\, M \times \PP^1$ as in Section \ref{sec3}, and the projections
$p$ and $q$ as in \eqref{pq}. Let $(E \,, \nabla_E)$ be
a flat complex vector bundle on $M$. Since $S^{0,1}$ in \eqref{dds} contains
$p^\ast T_\CC M$, there is a unique flat partial connection $p^\ast \nabla_E$ on
the smooth vector bundle $p^\ast E$ in the direction of $S^{0,1}$ such that
\[
  (p^\ast \nabla_E)(p^\ast s) \,=\, p^\ast(\nabla_E s)
\]
for every smooth section $s$ of $E$, where the right--hand side is considered as a section of $(S^{0,1})^\ast \tensor p^\ast E$ via the inclusion
\[
  p^\ast T_\CC^\ast M \tensor p^\ast E \,\longhookrightarrow\, (S^{0,1})^\ast \tensor p^\ast E \, .
\]
Similarly, for the holomorphic tangent bundle $(T\PP^1 = T^{1,0}\PP^1 \,,
\dbar_{T\PP^1})$ of $\PP^1$, there is a unique flat partial connection
$q^\ast \dbar_{T\PP^1}$ on the smooth vector bundle $q^\ast T\PP^1$ in the
direction of $S^{0,1}$ such that
\[
  (q^\ast \dbar_{T\PP^1})(q^\ast s) \,=\, q^\ast(\dbar_{T\PP^1} s)
\]
for every smooth section $s$ of $T\PP^1$.

Consider the trivial action of $\SU(2)$ on $M$ and the standard action of
$\SU(2)$ on $\PP^1$ (the left--translation action on $\SU(2)/\U(1)\,=\,
\PP^1$). They together define the diagonal action of $\SU(2)$ on $X$.
Then both the smooth vector bundles $p^\ast E$ and $q^\ast T\PP^1$
are equipped with natural equivariant actions of $\SU(2)$. Define the
vector bundle
\[
  F \,:=\, p^\ast E \oplus q^\ast T\PP^1\, .
\]

Now consider the $\SU(2)$--equivariant extension
\[
  0 \,\longrightarrow\, p^\ast E \,\stackrel{\iota}{\longrightarrow}\, F \,\stackrel{\pi}{\longrightarrow}\, q^\ast T\PP^1 \,\longrightarrow\, 0
\]
of smooth vector bundles on $X$, where $\iota$ and $\pi$ respectively are the
natural inclusion and projection maps.

\begin{proposition} \label{correspondence}
There is a bijective correspondence between the following two:
\begin{itemize}
\item[(i)]  flat sections $\phi$ of $(E \,, \nabla_E)$;
\item[(ii)] $\SU(2)$--invariant flat partial connections $\nabla_F$ on $F$ in the direction of $S^{0,1}$ such that
\begin{equation} \label{extension}
  0 \,\longrightarrow\, (p^\ast E \,, p^\ast \nabla_E) \,\stackrel{\iota}{\longrightarrow}\, (F \,, \nabla_F) \,\stackrel{\pi}{\longrightarrow}\, (q^\ast T\PP^1 \,, q^\ast \dbar_{T\PP^1}) \,\longrightarrow\, 0
\end{equation}
is an extension of $S^{0,1}$--partially flat vector bundles.
\end{itemize}
\end{proposition}

\begin{proof}
Fix a nonzero $\SU(2)$--invariant section
\[
  \alpha \,\in\, C^\infty(\PP^1 \,, (T^{0,1} \PP^1)^\ast \tensor (T^{1,0} \PP^1)^\ast)
  \,=\, \A^{0,1}(\PP^1 \,, (T^{1,0} \PP^1)^\ast) \, .
\]
Two such sections differ by an $\SU(2)$--invariant complex--valued function,
which must be constant due to the transitivity property of the $\SU(2)$--action
on $\PP^1$. Therefore, $\alpha$ is unique up to a constant factor.

Given a flat section $\phi$ of $(E \,, \nabla_E)$, define a partial connection
on $F$ in the direction of $S^{0,1}$
\[
  \nabla_F \,:=\, \begin{pmatrix}
 p^\ast \nabla_E & p^\ast \phi \tensor q^\ast \alpha \\
0  & q^\ast \dbar_{T\PP^1}
\end{pmatrix}
\]
with respect to the decomposition $F\,=\, p^\ast E \oplus q^\ast T\PP^1$.

Note that $p^\ast \phi \tensor q^\ast \alpha$ is a section of $\A^{0,1}(X \,,
\Hom(q^\ast T\PP^1 \,, p^\ast E))$. Clearly, $\nabla_F$ is $\SU(2)$--invariant,
and we have an extension as in \eqref{extension}. From the given condition that
$\nabla_E(\phi)\,=\, 0$ it follows that $\nabla_F$ is flat.

Conversely, given an $\SU(2)$--invariant flat partial connection $\nabla_F$ on
$F$ in the direction of $S^{0,1}$ as in (ii), it can be written as
\[
 \nabla_F \,=\, \begin{pmatrix}
 p^\ast \nabla_E & \beta \\
 0  & q^\ast \dbar_{T\PP^1}
  \end{pmatrix}
\]
in terms of the decomposition $F \,= \, p^\ast E \oplus q^\ast T\PP^1$, where
$\beta$ is an $\SU(2)$--invariant section of $\A^{0,1}(X \,, \Hom(q^\ast
T\PP^1 \,, p^\ast E))$. We have
\begin{gather*}
  \A^{0,1}(X \,, \Hom(q^\ast T\PP^1 \,, p^\ast E))
  \,=\, (S^{0,1})^\ast \tensor p^\ast E \tensor q^\ast (T^{1,0}\PP^1)^\ast \\
\cong\, \big(p^\ast(T_\CC^\ast M \tensor E) \tensor q^\ast (T^{1,0}\PP^1)^\ast\big)
 \,\oplus\, \big(p^\ast E \tensor q^\ast((T^{0,1} \PP^1)^\ast \tensor
(T^{1,0}\PP^1)^\ast)\big) \, .
\end{gather*}
One can see that the $\SU(2)$--invariant part of the first summand is zero by
restricting an $\SU(2)$--invariant element to the $\PP^1$--fibers and observing
that every section of $(T^{1,0} \PP^1)^\ast$ has to vanish at some point and
then by the transitivity of the $\SU(2)$--action it has to vanish everywhere.
Since $\alpha$ is a non--vanishing $\SU(2)$--invariant section of $(T^{0,1}
\PP^1)^\ast \tensor (T^{1,0}\PP^1)^\ast$, it follows that $\beta \,=\, p^\ast
\phi \tensor q^\ast \alpha$ for a unique smooth section
$\phi$ of $E$ (the section $\alpha$ was defined earlier). The flatness of
$\nabla_F$ then implies that $\phi$ is a flat section of $(E \,, \nabla_E)$.

The above two constructions are clearly inverses of each other.
\end{proof}

Let $\sigma$ be a positive real number. Define $g_\sigma$ to be the Hermitian
metric on $X$ with associated $(1 \,, 1)$--form
\[
  \Omega_\sigma \,:=\, p^\ast \omega_M - \sqrt{-1} \, \sigma q^\ast \omega_{\PP^1} \, ;
\]
it is a Gauduchon metric on $X$ because $g_M$ is an affine Gauduchon metric on
$M$. The degree of an $S^{0,1}$--partially flat vector bundle $E$ on $X$ with
respect to $g_\sigma$ will be denoted by $\deg_\sigma(E)$.

\begin{lemma} \label{lemma} \mbox{}
\begin{enumerate}
\item[(i)] If $E$ is a flat vector bundle over $M$, then
\[
  \deg_\sigma(p^\ast E) \,=\, n \sigma\cdot \deg(E) \, .
\]
\item[(ii)] If $V$ is a holomorphic vector bundle over $\PP^1$, then
\[
  \deg_\sigma(q^\ast V) \,=\, 2 \pi \cdot\vol(M) \deg(V) \, .
\]
\end{enumerate}
\end{lemma}

\begin{proof}
If $E$ is a flat vector bundle over $M$ and $h$ is a smooth Hermitian metric on $E$, we have
\[ \begin{split}
  \deg_\sigma(p^\ast E)
  &\,=\, \int_X \frac{c_1(p^\ast E \,, p^\ast h) \wedge \Omega_\sigma^n}{p^\ast \nu} \\
  &\,=\, - \sqrt{-1} \, n \sigma \int_X \frac{p^\ast c_1(E \,, h) \wedge p^\ast \omega_M^{n-1} \wedge q^\ast \omega_{\PP^1}}{p^\ast \nu} \\
  &\,=\, n \sigma \int_X p^\ast\left(\frac{c_1(E \,, h) \wedge \omega_M^{n-1}}{\nu}\right) \wedge q^\ast \omega_{\PP^1} \\
  &\,=\, n \sigma \int_M \frac{c_1(E \,, h) \wedge \omega_M^{n-1}}{\nu} \\
  &\,=\, n \sigma \cdot \deg(E)
\end{split} \]
since $\int_{\PP^1} \omega_{\PP^1} = 1$, thus proving (i).

For (ii), let $h$ be a smooth Hermitian metric on $V$. Then we have
\[ \begin{split}
  \deg_\sigma(q^\ast V)
  &\,=\, \int_X \frac{c_1(q^\ast V \,, q^\ast h) \wedge \Omega_\sigma^n}{p^\ast \nu} \\
  &\,=\, \int_X \frac{-2 \pi \sqrt{-1} \, q^\ast c_1(V \,, h) \wedge p^\ast \omega_M^n}{p^\ast \nu} \\
  &\,=\, 2 \pi \int_X q^\ast c_1(V \,, h) \wedge p^\ast\left(\frac{\omega_M^n}{\nu}\right) \\
  &\,=\, 2 \pi \left(\int_M \frac{\omega_M^n}{\nu}\right) \left(\int_{\PP^1} c_1(V \,, h)\right) \\
  &\,=\, 2 \pi \cdot \vol(M) \deg(V) \, ,
\end{split} \]
which proves (ii). Note that our definition of the first Chern form on $X$
imitates the definition of the first Chern form on affine manifolds given in
\cite{Lo09} and thus differs from the usual definition on complex manifolds,
which accounts for the factor $- 2 \pi \sqrt{-1}$ in the second line.
\end{proof}

\begin{corollary}
For an extension
\[
  0 \,\longrightarrow\, (p^\ast E \,, p^\ast \nabla_E) \,\stackrel{\iota}{\longrightarrow}\, (F \,, \nabla_F) \,\stackrel{\pi}{\longrightarrow}\, (q^\ast T\PP^1 \,, q^\ast \dbar_{T\PP^1}) \,\longrightarrow\, 0
\]
as in Proposition \ref{correspondence}, we have
\[
  \deg_\sigma(F) \,=\, n \sigma \cdot \deg(E) + 4 \pi \cdot\vol(M) \, .
\]
\end{corollary}

\begin{proof}
As in \cite[(23)]{Lo09}, we have
\[
  \deg_\sigma(F) \,=\, \deg_\sigma(p^\ast E) + \deg_\sigma(q^\ast T\PP^1) \, .
\]
The corollary then follows from Lemma \ref{lemma} and the fact that
$\deg(T\PP^1) \,=\, 2$.
\end{proof}

Using these formulas, and the correspondence between $S^{0,1}$--partially flat
vector bundles over $X\,= \,M \times \PP^1$ and holomorphic vector bundles over
$M^\CC \times\PP^1$ which are constant along the fibers of $M^\CC \,=\, TM
\longrightarrow M$, the following results from \cite{GP94b} immediately carry over to our situation. (See \cite[Theorem 4.9, Propositions 3.2, 3.11]{GP94b}.)

\begin{proposition} \label{stable}
Let $(E \,, \phi)$ be a flat pair on $M$ such that $E$ is not the trivial flat
line bundle, and let $(F \,, \nabla_F)$ be the
$\SU(2)$--equivariant $S^{0,1}$--partially flat vector bundle over $X$
corresponding to $(E \,, \phi)$ by Proposition \ref{correspondence}. Let the
real numbers $\sigma$ and $\tau$ be related by
\[
  \sigma \,=\, \frac{4 \pi \cdot \vol(M)}{n \, (\rank(E) + 1) \, \tau - n \cdot \deg(E)} \, .
\]
Then $(E \,, \phi)$ is $\tau$--stable if and only if $\sigma > 0$ and $F$ is stable with respect to $g_\sigma$.
\end{proposition}

\begin{proposition} \label{hermitian-einstein}
Let $(E \,, \phi)$ be a flat pair over $M$, and let $(F \,, \nabla_F)$ be the
$\SU(2)$--equivariant $S^{0,1}$--partially flat vector bundle over $X$
corresponding to $(E \,, \phi)$ by Proposition \ref{correspondence}.
\begin{enumerate}
\item[(i)] There is a bijective correspondence between the Hermitian metrics on
$E$ and the $\SU(2)$--invariant Hermitian metrics on $F$.

\item[(ii)] If the real numbers $\sigma$ and $\tau$ are related by
\[
  \sigma \,=\, \frac{4 \pi \cdot \vol(M)}{n \, (\rank(E) + 1) \, \widehat \tau - n \cdot \deg E} \, , \quad
  \text{where } \widehat \tau \,=\, \frac{\tau \cdot \vol(M)}{2} \, ,
\]
then a Hermitian metric $h$ on $E$ satisfies the $\tau$--vortex equation if and only if the Hermitian metric on $F$ corresponding to $h$ by (i) is a Hermitian--Einstein metric with respect to $g_\sigma$.
\end{enumerate}
\end{proposition}

The proofs of \cite[Theorem 4.9, Propositions 3.2, 3.11]{GP94b} can be applied to our situation by replacing the compact complex manifolds $X$ and $X \times \PP^1$ in \cite{GP94b} by the complex manifolds $M^\CC$ and $M^\CC \times \PP^1$, respectively. Note that although the latter manifolds are not compact, the proofs still go through because the degrees of holomorphic vector bundles over $M^\CC$ (respectively, $M^\CC \times \PP^1$) which are constant along the fibers of $M^\CC \,=\, TM \longrightarrow M$ are computed using integration over the compact manifold $M$ (respectively, $X \,=\, M \times \PP^1$).

We are now in a position to prove Theorem \ref{main}.

\begin{proof}[Proof of Theorem \ref{main}]
Let $(E \,, \phi)$ be a $\widehat \tau$--stable flat pair on $M$.

If $E$ is the trivial line bundle equipped with the trivial connection, then
$\phi$ is an element of $\CC^\ast$. Also, the pair $(E \,,
\phi)$ is $\widehat \tau$--stable if and only if $\widehat\tau \,>\, 0$, or
equivalently, $\tau \,>\, 0$. Using this, it can be easily checked that a
solution to the $\tau$--vortex equation in this case is given by
\[
  h \,:=\, \frac{\tau}{|\phi|^2} \, h_0 \, ,
\]
where $h_0$ is the constant Hermitian metric on $E$ given by the absolute value
(with respect to the trivialization of $E$).

Henceforth, we will assume that $E$ is not the trivial flat line bundle.

By Proposition \ref{stable}, the $\SU(2)$--equivariant $S^{0,1}$--partially flat
vector bundle $F$ over $X$ corresponding to $(E \,, \phi)$ is stable with respect to
$g_\sigma$, where
\[
  \sigma \,=\, \frac{4 \pi \cdot \vol(M)}{n \, (\rank(E) + 1) \, \widehat \tau - n \cdot \deg(E)} \, .
\]
Therefore, by Theorem \ref{duy}, this $\SU(2)$--equivariant $S^{0,1}$--partially
flat vector bundle $F$ admits a Hermitian--Einstein metric $h$ with respect
to $g_\sigma$. By pulling back $h$ using each element of $\SU(2)$ and
then averaging these using the Haar measure on the compact group $\SU(2)$, we
can produce an $\SU(2)$--invariant Hermitian--Einstein metric on $F$. By Proposition \ref{hermitian-einstein}, this metric corresponds to a Hermitian metric on $E$ solving the $\tau$--vortex equation.
\end{proof}

Again using the correspondence between the $S^{0,1}$--partially flat vector
bundles
on $X \,=\, M \times \PP^1$ and the holomorphic vector bundles on $M^\CC \times
\PP^1$, the
methods from \cite{GP94b} also show that if a flat pair on $M$ admits a Hermitian metric satisfying the $\tau$--vortex equation, then it must be $\widehat \tau$--polystable. Therefore, Theorem \ref{main} has the following corollary.
\begin{corollary} \label{corollary}
Let $(M \,, D \,, \nu)$ be a compact connected special affine manifold equipped with an affine Gauduchon metric $g_M$, and let $(E \,, \phi)$ be a flat pair on $M$. Let $\tau$ be a real number, and let
\[
  \widehat \tau \,=\, \frac{\tau \cdot \vol(M)}{2} \, .
\]
Then $E$ admits a smooth Hermitian metric satisfying the $\tau$--vortex
equation if and only if it is $\widehat \tau$--polystable.
\end{corollary}

\end{document}